\documentclass[11pt,a4paper]{amsart}

\usepackage[T1]{fontenc}
\usepackage[utf8]{inputenc}
\usepackage{microtype}
\usepackage{geometry}
\usepackage{amsmath,amssymb,amsthm,mathtools}
\usepackage{enumitem}
\usepackage{hyperref}
\usepackage{booktabs}

\makeatletter
\g@addto@macro\@adminfootnotes{%
  \ifx\@empty\addresses\else
    \@footnotetext{\@setaddresses}%
    \global\let\addresses\@empty
  \fi
}
\makeatother

\hypersetup{
  colorlinks=true,
  linkcolor=blue,
  citecolor=blue,
  urlcolor=blue,
  pdftitle={Remarks on the monodromy conjecture for symmetric determinantal hypersurfaces},
  pdfauthor={Renjie Lyu and Xiping Zhang}
}
\setlist[itemize]{leftmargin=2em}
\setlist[enumerate]{leftmargin=2.4em}

\newtheorem{theorem}{Theorem}[section]
\newtheorem{conjecture}[theorem]{Conjecture}
\newtheorem{proposition}[theorem]{Proposition}
\newtheorem{lemma}[theorem]{Lemma}
\newtheorem{corollary}[theorem]{Corollary}
\newtheorem{remark}[theorem]{Remark}
\theoremstyle{definition}
\newtheorem{definition}[theorem]{Definition}
\theoremstyle{remark}

\newcommand{\C}{\mathbf C}
\newcommand{\PP}{\mathbf P}
\newcommand{\Sym}{\operatorname{Sym}}

\newcommand{\GL}{\operatorname{GL}}

\newcommand{\Fl}{\operatorname{Fl}}
\newcommand{\CQ}{\operatorname{CQ}}
\newcommand{\Tot}{\operatorname{Tot}}
\newcommand{\Bl}{\operatorname{Bl}}
\newcommand{\ord}{\operatorname{ord}}
\newcommand{\rank}{\operatorname{rank}}
\newcommand{\codim}{\operatorname{codim}}
\newcommand{\Spec}{\operatorname{Spec}}

\newcommand{\cC}{\mathcal C}
\newcommand{\cD}{\mathcal D}

\title{Remarks on the monodromy conjecture for symmetric determinantal hypersurfaces.}
\author{Renjie Lyu}
\address{Mathematical School of Science, Xiamen University, Xiamen, 361005, China.} 
\email{r.lyu@xmu.edu.cn}

\author{Xiping Zhang}
\address{School of Mathematical Sciences, Tongji University, Shanghai, China.}
\email{xzhmath@gmail.com}


\begin{document}
\begin{abstract}
We compute the local and global Denef--Loeser topological zeta functions
associated to the determinant on the space of symmetric matrices. It verifies
the strong monodromy conjecture of the symmetric determinant function.
The computation relates to the geometry of complete quadrics and 
wonderful embeddings.
\end{abstract}
\maketitle
\setcounter{tocdepth}{1}
\tableofcontents

\section{Introduction}
\label{sec:introduction}

Consider a nonconstant complex polynomial \(f\in \C[x_1,\ldots, x_n]\). 
Let \(\cD=\C[x_i, \frac{\partial}{\partial x_j}]\) be the Weyl algebra 
of differential operators. The \emph{Bernstein-Sato} polynomial of \(f\), 
also called the \emph{\(b\)-function}, is the monic polynomial \(b_f(s)\) 
of minimal degree satisfying
\begin{equation}
\label{eqs:Bernstein-Sato-poly}
  P(s)f^{s+1}=b_f(s)f^s
\end{equation}
for some differential operator $P(s)\in \cD[s]$. If \(f\colon (\C^n, x)\to (\C, 0)\)
is a holomorphic germ at a point \(x\in \C^n\), the local \(b\)-function 
\(b_{f,x}(s)\) is the monic polynomial of minimal degree satisfying 
an equation similar to~\eqref{eqs:Bernstein-Sato-poly} where \(P(s)\) 
is a local differential operator defined on a small neighbourhood 
of \(x\). 

The roots of the \(b\)-function measure the hypersurface 
singularities of \(\{f=0\}\) in the following two aspects.
Bernstein~\cite{Bernstein} used the equation~\eqref{eqs:Bernstein-Sato-poly} to illustrate 
that the negative integer translates of the roots of \(b_f(s)\) control 
the poles of the meromorphic continuation of the Archimedean zeta function
\begin{equation}
\label{eqs:arch-zeta-fun}
  I_\varphi(s)=\int_{\C^n} |f(x)|^{2s}\varphi(x)dV(x), \quad \varphi(x)\in C^{\infty}_c(\C^n)
\end{equation}
Kashiwara's theorem~\cite{Kashiwara} shows that the roots 
\(\{\alpha \mid b_f(\alpha)=0\}\) are negative rational numbers,
and the elements of \(\{\operatorname{exp}(2\pi i\alpha) \mid b_f(\alpha)=0\}\) are 
eigenvalues of the monodromy action on the cohomology of Milnor fibres 
at points in the hypersurface \(\{f=0\}\)~\cite{KashiwaraVanishing,Malgrange}.

Let \(\pi\colon Y\to \C^n\) be an embedded resolution of the hypersurface 
\(D\coloneqq f^{-1}(0)\subset \C^n\), and let \(\{E_i\}_{i\in S}\) be the 
irreducible components of the reduced scheme \(\pi^{-1}(D)_{red}\) with normal crossings.
Write
\[
  \operatorname{div}(f\circ \pi)=\sum_{i\in S}N_iE_i,
  \qquad
  \mathrm{div}(\mathrm{Jac}(\pi))=\sum_{i\in S}(\nu_i-1)E_i.
\]
The change of variables shows that the possible poles of the Archimedean 
zeta function \(I_\varphi(s)\) are contained in the set 
\(\{-\frac{\nu_i+k}{N_i} \mid i\in S,\ k\in\mathbf Z_{\geq 0}\}\)~\cite[Section~1.6]{DenefReport}.

The same numerical data also arise in non-Archimedean settings. 
Let \(p>0\) be a prime number, and let \(f\in \mathbf Q_p[x_1,\ldots,x_n]\setminus \mathbf Q_p\). 
Igusa’s \(p\)-adic zeta function is
\[
Z_{p}(f,\varphi,s)=\int_{\mathbf Q_p^n}|f(x)|_p^s\varphi(x)dx  
\]
where \(|\cdot|_p\) is the \(p\)-adic norm, \(dx\) is the Haar measure
on \(\mathbf Q_p^n\), and \(\varphi\) is locally constant with compact 
support on \(\mathbf Q_p^n\). Igusa’s zeta function is a rational 
function in the variable \(p^{-s}\), and an embedded resolution of \(\{f=0\}\) 
produces denominators of the forms \(1-p^{-\nu_i-N_is}\); 
see~\cite[Sections~1.1 and~1.3]{DenefReport}. Then the real parts of 
the candidate poles of \(Z_{p}(f,\varphi,s)\) belong to the set 
\(\{-\frac{\nu_i}{N_i} \mid i\in S\}\). For a polynomial \(f\) with 
rational coefficients, Igusa's original monodromy conjecture predicts 
that, for almost all primes \(p\), a pole \(s_0\) of its \(p\)-adic 
zeta function gives the monodromy eigenvalue 
\(\operatorname{exp}^{2\pi i\mathrm{Re}(s_0)}\) of the Milnor fibre 
at some point of the complex hypersurface \(\{f_\C=0\}\); see \cite{Igusa-b-fun}.

As a characteristic-zero geometric counterpart of Igusa's zeta function,
Denef and Loeser introduced the \emph{topological zeta function}~\cite{DenefLoeser}.
Recall the embedded resolution \(\pi\colon Y\to \C^n\) of \(f^{-1}(0)\)
and the divisors \(E_i\) with the numerical data \((\nu_i, N_i)\). 
The global topological zeta function is
\begin{equation}
\label{eqs:top-zeta-fun}
  Z^{\mathrm{top}}_f(s)
  =
  \sum_{I\subseteq S}
  \chi(E_I^\circ)
  \prod_{i\in I}\frac{1}{\nu_i+N_is}, \quad E_I^\circ=
  \left(\bigcap_{i\in I}E_i\right)
  \setminus \left(\bigcup_{j\notin I}E_j\right).
\end{equation}
Here $\chi$ denotes the compactly supported Euler characteristic.
For $x\in f^{-1}(0)$, the local topological zeta function is
\begin{equation}
\label{eqs:local-top-zeta-fun}
  Z^{\mathrm{top}}_{f,x}(s)
  =\sum_{I\subseteq S}
  \chi\bigl(E_I^\circ\cap \pi^{-1}(x)\bigr)
  \prod_{i\in I}\frac{1}{\nu_i+N_is}.
\end{equation}
These rational functions are independent of the chosen resolution~\cite{DenefLoeser}.

\begin{conjecture}
\label{conj:topological-monodromy}
Let \(f\in \C[x_1\ldots, x_n]\) be a nonconstant polynomial, and 
let $x\in f^{-1}(0)$ be any closed point.
\begin{enumerate}[label=\textup{(\roman*)}]
\item If \(\alpha\) is a pole of \(Z^{\mathrm{top}}_f(s)\), then
$\exp(2\pi i\alpha)$ is an eigenvalue of the monodromy action on 
the cohomology of the Milnor fibre at some point \(y\in \{f=0\}\).
\item If \(\alpha\) is a pole of \(Z^{\mathrm{top}}_f(s)\), then \(\alpha\)
is a root of \(b_f(s)\).
\end{enumerate}
For the local topological zeta function $Z^{\mathrm{top}}_{f,x}(s)$,
the corresponding assertions use $b_{f,x}(s)$ and the Milnor fibre
of the germ \(f\) at some point \(y\in f^{-1}(0)\) in a small neighbourhood of \(x\).
\end{conjecture}
The assertion \((ii)\) is called the \emph{strong monodromy conjecture}~\cite{Veys}
since it implies \((i)\). If the orders of poles and roots are taken into account, 
one may strengthen the second assertion to require
\[
b_{f}(s)Z^{\mathrm{top}}_{f}(s)\in\C[s].
\] 

The generic determinant, the Pfaffian, and the symmetric determinant 
are classical examples of relative invariants of prehomogeneous 
vector spaces, and their \(b\)-functions are known~\cite{SatoKimura,Kimura}. 
Docampo~\cite{Docampo} computed topological zeta functions of generic
square determinantal varieties. L\H{o}rincz--Raicu--Walther--Weyman~\cite{LRWW} 
proved the strong monodromy conjecture for maximal minors of generic 
\(m\times n\) matrices and sub-maximal Pfaffians of odd-sized 
skew-symmetric matrices. Chen--Zuo~\cite{ChenZuo} and Chen--Shi--Xu--Zuo~\cite{ChenShiXuZuo} 
obtained monodromy and holomorphy results for general generic 
determinantal ideals and Pfaffian ideals. Our initial aim
was to verify the strong monodromy conjecture for the symmetric
determinantal hypersurface, thereby complementing the literature. After 
finishing the first draft, we realized that the desired conjecture 
can be deduced from specializations of $p$-adic zeta functions 
and Igusa's monodromy conjecture of prehomogeneous vector spaces, 
which are explained at the end of the introduction. Nevertheless, 
in this paper, we obtain the exact formulas for the local and global 
topological zeta functions, which determine the explicit orders of poles.

Let $V$ be a complex vector space of dimension $n\geq2$, and set
\begin{equation}
\label{eq:symmetric-determinant}
  M_n=\Sym^2(V^*),\qquad f=\det\colon M_n\longrightarrow\C .
\end{equation}
We denote by $\Sigma_r=\{A\in M_n\mid\rank(A)\leq r\}$ for
$0\leq r\leq n$ the rank loci of the determinantal
hypersurface $D\coloneqq \Sigma_{n-1}=f^{-1}(0)$. Define the rational 
functions
\begin{align}
  F_{2m+1}(s)
  &=
  \frac{(2m+1)!!}
  {(s+1)\prod_{j=1}^{m}(2s+2j+1)}
  &&(m\geq0),                                      \label{eq:Fodd}\\
  F_{2m}(s)
  &=
  \frac{(2m-1)!!(s+2m+1)}
  {(s+1)\prod_{j=1}^{m}(2s+2j+1)}
  &&(m\geq1).                                      \label{eq:Feven}
\end{align}

\begin{theorem}
\label{thm:main}
The global topological zeta function of the symmetric determinant \(f\) is
\[
  Z^{\mathrm{top}}_f(s)=Z^{\mathrm{top}}_{f,0}(s)=F_n(s).
\]
For $x\in \Sigma_r \subseteq D$ of rank $r$, the local topological zeta function is
\[
  Z^{\mathrm{top}}_{f,x}(s)=F_{n-r}(s).
\]
The actual poles of \(Z^{\mathrm{top}}_{f,x}(s)\) are precisely
\begin{equation}
  \{-1\}\cup
  \left\{-\frac{2j+1}{2}\ \middle|\
  1\leq j\leq\left\lfloor\frac{n-r}{2}\right\rfloor\right\}.
  \label{eq:actual-poles}
\end{equation}
All these are simple poles. The local Bernstein--Sato polynomial is
\begin{equation}
  b_{f,x}(s)=
  \prod_{j=1}^{n-r}\left(s+\frac{j+1}{2}\right).
  \label{eq:local-b}
\end{equation}
Consequently, we have 
\[
b_{f,x}(s)Z^{\mathrm{top}}_{f,x}(s)\in\C[s]  
\]
which verifies the strong monodromy conjecture at every point of the
symmetric determinantal hypersurface. 
\end{theorem}

The embedded resolution of the determinantal hypersurface \(D\) is 
obtained by successive blow-ups along the rank strata
\(\{0\}=\Sigma_0\subset \Sigma_1\subset \ldots \subset \Sigma_{n-1}=D\).
The iterated strict transforms of \(\Sigma_r\) produce divisors 
\(E_r\) of the resolution. Gaube--Schober's description of the 
successive blow-ups~\cite{GaubeSchober} shows that the local structures 
of the iterated strict transforms of the strata are again symmetric determinantal 
singularities of smaller size. It follows that the coranks 
and codimensions of \(\Sigma_r\setminus \Sigma_{r-1}\) determine the 
numerical data \((N_r, \nu_r)\) of the divisors \(E_r\), leading to 
the candidate poles of the topological zeta function; see Corollary~\ref{cor:multi-discrepancy}. 
The transversal normal slices to the rank strata, obtained through 
discussion with AI, provide a more direct argument for the above 
assertion; see Lemma~\ref{lem:normal-slice}.

The computation of the exact topological zeta function uses the 
geometry of complete quadrics. The successive blow-ups of the projective
rank loci appear in Vainsencher's work on complete collineations and 
complete quadrics~\cite{VainsencherCompleteCollineations,VainsencherCompleteQuadrics}. 
This projective resolution is naturally a wonderful embedding 
of the space of smooth projective quadrics, which links to the theory 
of complete symmetric varieties developed by De Concini--Procesi~\cite{DeConciniProcesi}
and De Concini, Goresky, MacPherson, and Procesi~\cite{DeConciniGoreskyMacPhersonProcesi}.
The affine blow-up naturally admits a line bundle structure over the
projective blow-up at each stage; see Corollary~\ref{cor:affine-CQ}.
Then the geometry of the boundary orbits of the wonderful embedding 
gives a combinatorial characterization of the Euler characteristics 
needed to evaluate the topological zeta function~\eqref{eqs:top-zeta-fun}.

The specialization theorem of Denef and Loeser, as stated 
in~\cite[Theorem~4.3.2]{DenefReport}, shows that every topological 
pole occurs, at almost all places of good reduction, as an actual 
pole of the corresponding $p$-adic zeta function over suitable finite 
unramified extensions. The theorem of 
Kimura--Sato--Zhu~\cite[Theorem~3.1]{KimuraSatoZhu} therefore applies:
every real pole of the $p$-adic zeta functions associated with a 
reduced regular irreducible prehomogeneous vector space is a root of 
its $b$-function. Symmetric determinants belong to this class. This 
arithmetic argument yields the strong monodromy conjecture for 
our cases without identifying the precise values of the poles.
The contribution of the present work is the closed formulas of the 
topological zeta functions, and their geometric connections 
to the complete quadrics. 

In an independent work, Doan~\cite{Doan} computes the local topological 
zeta functions in a more general setting, namely the generic symmetric 
determinantal ideals and catalecticant determinantal ideals, and 
verifies the associated strong monodromy conjecture.

\subsection*{Organization}
Section~\ref{sec:affine-reso} discusses the rank strata resolution and 
its numerical data. Section~\ref{sec:zeta} relates the resolution to 
complete quadrics and computes the topological zeta function,
and completes verification of the strong monodromy conjecture.

\subsection*{Acknowledgements.}
The authors acknowledge the assistance of ChatGPT~5.5~Pro in
exploring the connection between the affine resolution and the
geometry of complete quadrics, and in the calculation of Euler 
characteristics. The manuscript is written by humans. The authors 
are responsible for the mathematical arguments and conclusions.

\section{Embedded resolution by rank strata}
\label{sec:affine-reso}
For \(0\le r\le n-1\), let 
\[
\Sigma_r=\{A\in M_n\mid \rank(A)\leq r\}
\]
be the locus of symmetric matrices of rank at most \(r\). As a subscheme, 
\(\Sigma_r\) is defined by the vanishing of all \(r+1\)-minors of 
\(n\times n\) symmetric matrices. For \(r=n-1\), \(\Sigma_{n-1}\) is 
the determinantal hypersurface \(D\subset M_n\). The rank loci
form natural strata
\[
  \{0\}=\Sigma_0\subset \Sigma_1\subset \cdots \subset \Sigma_{n-1}=D.
\] 
The singular locus of \(\Sigma_r\) is the sub-stratum \(\Sigma_{r-1}\).


\begin{theorem}
\cite[Theorem~B]{GaubeSchober}
\label{thm:inductive-blowup}
  Let \(r\in \mathbf Z_{+}\) be any positive integer with \(r < n\). 
  The following sequence of blow-ups is an embedded resolution of
  singularities for the stratum \(\Sigma_{r}\subset M_n\),
  \[
  X_{r}\xrightarrow{\pi_r} \cdots \xrightarrow{\pi_2} X_1\xrightarrow{\pi_1} X_0\coloneqq M_n
  \]
where \(\pi_\alpha\colon X_\alpha\to X_{\alpha-1}\) is the blow-up
with center the strict transform of \(\Sigma_{\alpha-1}\) in 
\(X_{\alpha-1}\). In particular, the composition 
\begin{equation}
\label{eqs:log-res}
  \pi = \pi_{1}\circ \cdots \circ \pi_{n-1}\colon X = X_{n-1}\to M_n
\end{equation}
is an embedded resolution of singularities for the determinantal hypersurface \(D\subset M_n\).
\end{theorem}

We set the following notation. For \(0\le r\le n-1\), set \(E^{(0)}_{r}\coloneqq \Sigma_r\). 
Inductively, for \(1\le \alpha\le n-1\), let \(E^{(\alpha)}_{r}\) be the 
strict transform of \(E^{(\alpha-1)}_{r}\subset X_{\alpha-1}\) in the 
blow-up \(X_\alpha\) of \(X_{\alpha-1}\) with the center 
\(E^{(\alpha-1)}_{\alpha-1}\). Write \(E_r\) for the final strict 
transform \(E^{(n-1)}_{r}\subset X_{n-1}\) over \(\Sigma_r\). Then
the divisors \(E_0, E_1, \ldots, E_{n-1}\) are smooth, meet with 
simple normal crossings, and satisfy
\[
  \pi^{-1}(D)_{\mathrm{red}} = \bigcup_{r=0}^{n-1}E_r.
\]

\begin{corollary}
\label{cor:multi-discrepancy}
Let \(\pi\colon X\to M_n\) be the embedded resolution~\eqref{eqs:log-res}. 
For each $0\leq r\leq n-1$, the multiplicity and log discrepancy of $E_r$ are
\begin{equation}
\label{eq:numerical-data}
  N_r=\ord_{E_r}(f\circ\pi)=n-r,
  \qquad
  \nu_r=1+\ord_{E_r}(K_{X/M})
  =\frac{(n-r)(n-r+1)}{2}.
\end{equation}
\end{corollary}
  
\begin{proof}  
The computation uses the inductive local structures of the strict 
transforms in the successive blow-ups of Theorem~\ref{thm:inductive-blowup}. 

Let \((x_{ij})\) be coordinates of the space of \(n\times n\) symmetric 
matrices. Denote by \(f =\det (x_{ij})\) the determinant function. 
The first blow-up \(X_1\) admits a standard affine open cover. Let 
\(X_{ij}\) denote the open chart with the coordinates \(\tilde{x}_{ij}=x_{ij}\) 
and \(\tilde{x}_{k\ell}=\tilde{x}_{ij}x_{k\ell}, (k, \ell)\neq (i, j)\). 
On the \(X_{ii}\)-chart, we have
\[
  f\circ \pi_1|_{X_{ii}}=\tilde{x}_{ii}^{n}\cdot f'.
\]
The exceptional divisor \(E_0^{(1)}\) is given by \(\tilde{x}_{ii}=0\). 
After suitable row and column operations, one can choose coordinates 
\(y_{ij}\) with \(1\le i\leq j \le n-1\) that are transversal to 
\(\tilde{x}_{ii}\) such that 
\[
  f'=\det (y_{ij})
\] 
is the determinant of the \((n-1)\times (n-1)\) generic symmetric matrix 
\((y_{ij})\). Moreover, the ideal of the strict transform \(E^{(1)}_{r}\) 
for \(r\ge 1\) is given by the \(r\)-minors of \((y_{ij})\). Therefore
the subsequent blow-ups along \(E^{(r)}_r\) form an inductive process 
on the \((n-r)\times (n-r)\)-symmetric matrices.

On the \(X_{ij}\)-chart we have
  \[
  f\circ \pi_1|_{X_{ij}}=\tilde{x}_{ij}^{n}\cdot f'
  \]
where \(\tilde{x}_{ij}=0\) defines the exceptional divisor \(E_0^{(1)}\).
We may shrink the \(X_{ij}\)-chart such that \(E^{(1)}_1\cap X_{ij}=\emptyset\).
After suitable row and column operations, one can choose coordinates 
\(y_{ij}\) with \(1\le i\leq j \le n-2\) that are transversal to 
\(\tilde{x}_{ij}\) such that
  \[
  f'=\phi\cdot \det (y_{ij})
  \] 
where \(\phi\) is a unit and \((y_{ij})\) is a generic \((n-2)\times (n-2)\)
symmetric matrix. Then the ideal of the strict transform \(E^{(2)}_r\) for 
\(r\ge 2\) is given by the \((r-1)\)-minors of \((y_{ij})\). This open chart 
is unaffected by the blow up along \(E^{(1)}_{1}\). 

We give a coordinate-free summary of the above results. Each
blow-up \(X_\alpha\) admits an affine open cover which consists of 
open charts of two types, called the diagonal chart \(U_d\) and the 
off triangle chart \(U_l\). The stratified strict transforms
\[
  E^{(\alpha)}_\alpha\subset E^{(\alpha)}_{\alpha+1}\subset \cdots \subset E^{(\alpha)}_{n-1}
\] 
on the diagonal chart \(U_d\) are represented by the filtration
\begin{equation}
\label{eqs:stratified-normal-slice}
    \mathbf{C}^{m_\alpha}\times (Z^{(n-\alpha)}_{0}\subset Z^{(n-\alpha)}_{1}\subset \cdots \subset Z^{(n-\alpha)}_{n-\alpha-1})
\end{equation}
where \(m_\alpha=\dim E^{(\alpha)}_\alpha= \frac{\alpha(2n-\alpha+1)}{2}\) 
and \(Z^{(n-\alpha)}_{r}\) is the locus of \((n-\alpha)\times (n-\alpha)\) 
symmetric matrices of rank at most \(r\).

On the lower triangle chart \(U_l\), the lowest rank stratum 
\(E^{(\alpha)}_\alpha\) does not meet with \(U_l\). And the 
stratified strict transforms
  \[
  E^{(\alpha)}_{\alpha+1}\subset \cdots \subset E^{(\alpha)}_{n-1}
  \] 
correspond to the filtration
  \[
    \mathbf{C}^{m_{\alpha+1}}\times (Z_{0}^{(n-\alpha-1)}\subset Z_{1}^{(n-\alpha-1)}\subset \cdots \subset Z^{(n-\alpha-1)}_{n-\alpha-2})
  \]
where \(m_{\alpha+1}=\dim E^{(\alpha)}_{\alpha+1}=\frac{(\alpha+1)(2n-\alpha+2)}{2}\)
and \(Z^{(n-\alpha-1)}_{r}\) is the locus of \((n-\alpha-1)\times (n-\alpha-1)\)
symmetric matrices of rank at most \(r\).

In conclusion, every blow-up \(\pi_{\alpha+1}\) creates an exceptional 
divisor \(E^{(\alpha+1)}_{\alpha}\). The determinant function \(f\) has
order \(n-\alpha\) on a generic point \(\eta\in E^{(\alpha)}_{\alpha}\).
Hence 
\[
\pi_{\alpha+1}^*(E_{n-1}^{(\alpha)}) = E^{(\alpha+1)}_{n-1}+(n-\alpha)E^{(\alpha+1)}_{\alpha} 
\]
For \(r < \alpha\), the strict transforms \(E^{(\alpha)}_r\), as divisors 
in \(X_\alpha\), are smooth and transversely intersect with the center 
\(E^{(\alpha)}_\alpha\). Hence
\[
 \pi_{\alpha+1}^*(E^{(\alpha)}_r) = E^{(\alpha+1)}_r, \quad r < \alpha. 
\]
The relative canonical divisor of a blow-up along a smooth center of 
codimension \(c\) has discrepancy \(c-1\) on the exceptional divisor. 
Hence
\[
\mathrm{ord}_{E^{(\alpha+1)}_\alpha} \mathrm{Jac}(\pi_{\alpha+1}) = \frac{(n-\alpha+1)(n-\alpha)}{2}-1.  
\]
As a consequence,
\[
  \pi^*(D)=\sum_{r=0}^{n-1} (n-r)E_r, \qquad K_{X/M_n}=\sum_{r=0}^{n-1} \frac{(n-r+2)(n-r-1)}{2} E_r.
\]
\end{proof}



\paragraph{A transverse normal form}
Chatgpt 5.5 pro suggests the following 
transverse normal form of all rank strata near a matrix of fixed rank.
This normal slice, which is analogous to the inductive stratified strict 
transforms, simplifies the proof of Corollary~\ref{cor:multi-discrepancy}.


\begin{lemma}
\label{lem:normal-slice}
Let $0\leq r\leq n-1$, and let $A_0\in\Sigma_r\setminus\Sigma_{r-1}$.  
On a Zariski neighbourhood of $A_0$, there is an isomorphism of the 
rank-stratified space:
\begin{equation}
\label{eqs:stratifed-normal-slice}
  (M_n;\Sigma_r,\Sigma_{r+1},\ldots,\Sigma_n)
  \simeq
  (\text{a smooth factor})\times
  \bigl(\Sym^2(\C^{n-r})^*;
  \Sigma_0^{(n-r)},\ldots,\Sigma_{n-r}^{(n-r)}\bigr).
\end{equation}
Under this identification, the determinant function $f$ is locally 
a unit times the determinant of the $(n-r)\times (n-r)$ symmetric matrices.
\end{lemma}

\begin{proof}
After changing basis, we can choose an open neighbourhood of \(A_0\) 
on which the leading $r\times r$ blocks are invertible. A symmetric 
matrix there has the form
\[
  A=
  \begin{pmatrix}B&C\\ C^t&D\end{pmatrix},
  \qquad
  B^t=B\in\GL_r,\quad D\in\Sym^2(\C^q)^*.
\]
The congruence transformation by the matrix \(T=
\begin{pmatrix}
I_r&-B^{-1}C\\
0&I_{n-r}
\end{pmatrix}\) gives
\[
T^tAT=
\begin{pmatrix}
B&0\\
0&D-C^tB^{-1}C
\end{pmatrix}
\]
where \(S=D-C^tB^{-1}C\in \Sym^2(\C^{n-r})^*\). The change of variables
\[
  (B,C,D)\longmapsto(B,C,S),\qquad
  S=D-C^tB^{-1}C
\]
yields the isomorphism~\eqref{eqs:stratifed-normal-slice}.

By the transformation \(A\to T^tAT\) we have
\[
  \det(A)=\det(T)^2\det(A)=\det(B)\det(S).
\]
Since \(B\) is invertible, \(\det(B)\) gives a holomorphic unit in 
the analytic coordinates \((B, C, S)\). Thus \(f\) locally is a unit 
times the determinant of \((n-r)\times (n-r)\)-symmetric matrices.
\end{proof}

\begin{remark}
  The successive blow-ups \(\pi_\alpha\) for \(\alpha \le r\) do 
  not affect the smooth part \(\Sigma_r\setminus \Sigma_{r-1}\). At 
  any smooth point of \(\Sigma_r\), Lemma~\ref{lem:normal-slice} 
  identifies the normal slice with \(\Sym^2(\C^{n-r})^*\), and \(f\)
  restricts to the determinant function of degree \(n-r\) multiplied by
  a unit. Hence \(\codim_{M_n}\Sigma_r=\binom{n-r+1}{2}\), and \(f\)
  has order \(n-r\) along the exceptional divisor over \(\Sigma_r\),
  which leads to the multiplicity and log discrepancy of the divisor \(E_r\).
\end{remark}

\section{The topological zeta function}\label{sec:zeta}
The numerical data in Corollary~\ref{cor:multi-discrepancy} imply that
the candidate poles of \(Z^{\mathrm{top}}_f(s)\) have the form 
\(\{-\frac{n-r+1}{2}, 0\le r \le n-1\}\). This is sufficient to conclude
the strong monodromy conjecture of the symmetric determinantal hypersurface 
since the \(b\)-function of \(f\) is known; see Theorem~\ref{thm:b-function}.
The actual poles of \(Z^{\mathrm{top}}_f(s)\) depend on the cancellation
of the coefficients \(\chi(E^\circ_{I})\). In this section,
we compute the Euler characteristics using the theory of wonderful embeddings 
relating to the variety of nonsingular quadrics.

\subsection{Complete quadrics}

\begin{definition}
\label{def:wonderful-embed}
Let \(G\) be a connected algebraic group. Let \(X\) be a smooth, complete 
\(G\)-variety containing a dense open homogeneous \(G\)-subvariety 
\(X_0\). Then \(X\) is a wonderful embedding of \(X_0\) if 
\begin{itemize}
  \item \(X\setminus X_0\) is a union of finitely many \(G\)-stable 
  smooth divisors \(\{X_i\mid 1\leq i\leq k\}\), and the divisors \(X_i\) meet with 
  normal crossings;
  \item all \(G\)-orbits of \(X\) are indexed by subsets of \([k]\):
  for any \(I \subset [k]\), the orbit \(\mathcal{O}_I\) has the form
  \[
    \mathcal{O}_I=\bigcap_{i\in I} X_i-\bigcup_{j\notin I} X_j.
  \]
\end{itemize}
If a wonderful embedding of \(X_0\) exists, then it is unique up to \(G\)-equivariant isomorphism~\cite[Section~2.3]{CanJoyce}.  
\end{definition}
\begin{remark}
\label{rmk:orbits}
Through the definition, the \(G\)-orbits can recover the divisors.  
For any subset \(I\subset [k]\), the closure of the orbit 
\(\mathcal O_I\) is
\[
  \overline{\mathcal O_I}=\bigcup_{I\subset K\subset [k]} \mathcal O_K.
\] 
In particular, we have
\[
X_i  = \bigcup_{i\in K\subset [k]} \mathcal O_K, \quad i\in [k].
\]
\end{remark}

Let \(Q_0\subset \mathbf P(\mathrm{Sym}^2(V^*))\) be the parameter 
space of nonsingular quadric hypersurfaces in \(\mathbf P^{n-1}\cong \mathbf P(V)\).
The group \(\mathrm{SL}_n\) acts on \(Q_0\) by congruence
\[
(g, A)\mapsto gAg^{t}, \quad \forall g\in \mathrm{SL}_n, \quad \forall A\in Q_0
\]
where the stabilizer group \(H\) is the normalizer of \(\mathrm{SO}_n\). 
The variety of complete quadrics \(\CQ(V)\) is the wonderful embedding 
of the symmetric space \(\mathrm{SL}_n/H\cong Q_0\)~\cite[Section~3.A]{ManivelEtAl}. The orbits of
\(\CQ(V)\) have the following geometric characterization. When 
\(I=\emptyset\), we set \(\mathcal O_{I}=Q_0\) as the open dense orbit.
For any subset \(I = \{i_1, \dots, i_{k-1}\}\subset [n-1]\) with
\(i_1 < i_2 < \cdots < i_{k-1}\), the orbit \(\mathcal O_I\) consists 
of all pairs \((F, Q)\) where
\begin{itemize}
\item \(F\) is a partial flag 
   \[
   0=F_0\subset F_1\subset\cdots\subset F_{k-1}\subset F_k=V,\quad
   \dim F_p=i_p, ~1\leq p\leq k-1;
   \]
\item \(Q\) is a collection of nondegenerate projective quadrics
in each quotient \(\mathbf P(F_p/F_{p-1})\).
\end{itemize}
For instance, when \(I=[n-1]\), \(\mathcal O_{I}\) is the complete 
flag variety \(\mathrm{Fl}_{(1,\dots,1)}(V)\) giving the unique
closed orbit in \(\CQ(V)\)~\cite[Section~2.4]{CanJoyce}.

\paragraph{Complete collineations}
In the paper~\cite{VainsencherCompleteQuadrics}, Vainsencher uses
the complete collineation map to construct the complete quadrics.
The complete collineation is the rational map
\[
\begin{aligned}
\mathbf P(\Sym^2 V^*)&\dashrightarrow \prod_{i=2}^{n-1} \mathbf P(\Sym^2(\wedge^i V^*)).\\
[A]&\rightarrow \bigl([\wedge^2A],\ldots,[\wedge^{n-1}A]\bigr).
\end{aligned}
\]
The variety of complete quadrics \(\CQ(V)\) is defined as the graph 
of this rational map.

A matrix \(A\in \Sym^2 V^*\) has \(\mathrm{rank}(A) \le r\) if and 
only if \(\wedge^{r+1} A=0\). Hence the indeterminacy locus of the
\((r+1)\)-th component of the complete collineation map is the 
projective rank locus
\[
    \Sigma_r^{\mathrm{proj}}
    =
    \{[A]\in \mathbf P(\Sym^2 V^*)\mid \rank(A)\le r\}, \quad 1\le r\le n-1.
\]
It is proved in \cite[Theorem 6.3]{VainsencherCompleteQuadrics}
that the graph of the complete collineation map is isomorphic to the
successive blow-ups along the strict transforms of
\(\Sigma_1^{\mathrm{proj}},\Sigma_2^{\mathrm{proj}},\dots,\Sigma_{n-2}^{\mathrm{proj}}.\) This is similar to the affine resolution~\eqref{eqs:log-res}. Let 
\begin{equation}
\label{eqs:proj-res}
  \CQ(V)= Y_{n-1}\xrightarrow{\rho_{n-1}} \cdots \rightarrow Y_2\xrightarrow{\rho_2} Y_1 = \mathbf P(M_n)
\end{equation}
be the projective successive blow-ups. The resulting divisors 
\(D_1, \dots, D_{n-1}\) in \(\CQ(V)\) over the projective rank loci 
\(\Sigma_1^{\mathrm{proj}}, \dots, \Sigma_{n-1}^{\mathrm{proj}}\) are 
the boundary divisors in the sense of Definition~\ref{def:wonderful-embed}.

Let \(L\) be the tautological line bundle \(\mathcal O_{\mathbf P}(-1)\)
on \(\mathbf P(M_n)\). The blow-up \(\Bl_{0} M_n\) of \(M_n\) at the 
origin is isomorphic to the total space \(\Tot(L)\). The affine rank
loci \(\Sigma_r\) are the affine cones over \(\Sigma_r^{\mathrm{proj}}\) 
for \(1\le r\le n\). So the first step strict transform \(E^{(1)}_r\) 
is isomorphic to the total space \(\Tot(L|_{\Sigma_r^{\mathrm{proj}}})\) 
respectively, and \(E^{(1)}_0\) is the zero section of the projection 
\[
  q\colon\Bl_0(M)\cong \Tot(L)\to \mathbf P(M_n).
\]
Such line bundle structures are preserved in the successive blow-ups due 
to the following lemma, which leads to useful relations between the 
divisors \(E_0,E_1\dots, E_{n-1}\) in the affine resolution~\eqref{eqs:log-res} 
and \(\CQ(V), D_1, \dots, D_{n-1}\); see Corollary~\ref{cor:affine-CQ}.

\begin{lemma}
\label{lem:line-bundle}
Let $Z\subseteq Y$ be a closed subscheme and
$\tau\colon\widetilde Y=\Bl_Z Y\to Y$ be the blow-up. Let \(L\) be
a line bundle over \(Y\) and \(q\colon \Tot(L)\to Y\) the total space
projection. There are natural isomorphisms
\begin{equation}
  \Bl_{q^{-1}(Z)}\Tot(L)
  \cong
  \Tot(L)\times_Y\widetilde Y
  \cong
  \Tot(\tau^*L).
  \label{eq:line-bundle-blowup}
\end{equation}
\end{lemma}

\begin{proof}
The map $q$ is flat, and blow up commutes with flat base change
\cite[Lemma~31.33.3]{StacksBlowup}. This yields the first isomorphism.  The
second follows from the compatibility of relative spectra with base
change:
\[
  \Spec_Y\!\bigl(\Sym_{\mathcal O_Y}L^\vee\bigr)\times_Y\widetilde Y
  \cong
  \Spec_{\widetilde Y}\!\bigl(
  \Sym_{\mathcal O_{\widetilde Y}}(\tau^*L)^\vee\bigr).
\]
\end{proof}

\begin{corollary}
\label{cor:affine-CQ}
Let \(\rho\colon \CQ(V)\to \mathbf P(M_n)\) be the composition of the 
sequence of blow-ups~\eqref{eqs:proj-res}. Let 
\(L=\mathcal O_{\mathbf P}(-1)\), and let \(X\) be the log resolution
\eqref{eqs:log-res}. There is an isomorphism 
\[
X\cong\Tot(\rho^*L).
\] 
The divisor $E_0$ is a zero section of the projection 
\[
  q\colon X \to \CQ(V),
\]
and 
\begin{equation}
\label{eqs:line-bdl-div-orb}
  E_r \cong q^{-1}(D_r)\cong \Tot(\rho^* L|_{D_r})\qquad(1\leq r\leq n-1).
\end{equation}
\end{corollary}

\begin{proof}
The blow-up of \(M_n\) at the origin is \(\Bl_{0} M_n\cong \Tot(L)\).
For \(1\le r\le n-1\), the strict transform \(E^{(1)}_r\) of the 
affine rank locus \(\Sigma_r\) is isomorphic to the total space 
\(\Tot(L|_{\Sigma_r^{\mathrm{proj}}})\). Apply Lemma~\ref{lem:line-bundle} 
inductively to the blow-ups in~\eqref{eqs:proj-res}. At every step, 
\(q_\alpha\colon X_\alpha\to Y_\alpha\) is the line bundle projection 
of the pullback of \(L\). The divisor \(E^{(\alpha)}_0\) is the zero 
section of \(q_\alpha\), and \(E^{(\alpha)}_r=q_{\alpha}^{-1}(D^{(\alpha)}_r)\) 
where \(D^{(\alpha)}_r\) are the iterated strict transforms of \(\Sigma^{\mathrm{proj}}_{r}\).
\end{proof}

\subsection{Euler characteristics of boundary orbits}

An ordered partition of the integer $n$ is a tuple 
\(\mu=(\mu_1,\dots, \mu_k)\) of positive integers \(\mu_i\) satisfying 
\(\sum_{i=1}^{k} \mu_i=n\). There is a bijection
\[
\left\{\textrm{ordered partitions of}~ n\right\} \leftrightarrow \left\{\textrm{subsets of}~ [n-1]\right\}
\]
which assigns to an ordered partition \(\mu\) the subset
\[
  s_{\mu}=\{s_1, \ldots, s_{k-1}\},\quad s_a=\mu_1+\cdots+\mu_a.
\]
We denote by \(\mathcal O_{\mu}\) the orbit \(\mathcal O_{s_{\mu}}\) 
of the complete quadrics \(\CQ(V)\). 

Note that any \([A]\in \Sigma_r^{\mathrm{proj}}\setminus \Sigma_{r-1}^{\mathrm{proj}}\) 
is naturally associated with a flag \(0=F_0\subset \mathrm{Ker}(A)\cong \C^{n-r}\subset \C^{n}\). 
Hence the divisor \(D_r\) over \(\Sigma_r^{\mathrm{proj}}\) contains 
the orbit \(\mathcal O_{n-r}\) as its open dense subset. Moreover,
we have the relations
\[
  \mathcal O_{s_\mu}=D_{I_\mu}^\circ=\bigcap_{i\in I_\mu}D_{i}
  -\bigcup_{j\notin I_\mu}D_j,\quad I_{\mu}=\{n-s_1, \dots, n-s_{k-1}\}
\]
which is different from the usual index in Definition~\ref{def:wonderful-embed}. 

\begin{proposition}
\label{thm:orbit-euler}
For an ordered partition $\mu=(\mu_1,\ldots,\mu_k)$ of $n$, the Euler
characteristic is
\begin{equation}
  \chi(\mathcal O_\mu)=
  \begin{cases}
    \dfrac{n!}{\mu_1!\cdots\mu_k!},
    &\text{if every }\mu_a\in\{1,2\},\\[3mm]
    0,&\text{otherwise}.
  \end{cases}
  \label{eq:orbit-euler}
\end{equation}
Let \(q(\mu)=\#\{a \colon \mu_a=2\}\). Then in the nonzero case
\[
\chi(\mathcal O_\mu)=\frac{n!}{2^{q(\mu)}}.
\]
\end{proposition}

\begin{proof}
Recall the orbit \(\mathcal O_{\mu}\) parametrizes all partial flags
\begin{equation}
\label{eqs:flag-index}
   0=F_0\subset F_1\subset\cdots\subset F_{k-1}\subset F_k=V,\quad
   \dim F_a/F_{a-1}=\mu_a, ~1\leq a\leq k;
\end{equation}
and nondegenerate projective quadrics in \(\mathbf P(F_a/F_{a-1})\). 
Denote by \(\Fl_\mu(V)\) the flag variety of type~\eqref{eqs:flag-index}, 
and \(U_{\mu_a}\subset \mathbf P\bigl(\Sym^2(\C^{\mu_a})^*\bigr)\)
the space of nondegenerate projective quadratic forms of rank \(\mu_a\).
There is a locally trivial fibration
\begin{equation}
\label{eq:orbit-fibration}
  \prod_{a=1}^kU_{\mu_a}\longrightarrow \mathcal O_\mu
  \longrightarrow\Fl_\mu(V).
\end{equation}
Therefore we have
\[
\chi(\mathcal O_\mu) = \chi(\operatorname{Fl}_\mu(V)) \cdot \prod_{a=1}^k\chi(U_{\mu_a}).
\]
The Schubert-cell decomposition gives~\cite[Section~1.2]{BrionFlag} 
\[
  \chi\bigl(\Fl_\mu(V)\bigr)=
  \frac{n!}{\mu_1!\cdots\mu_k!}.
\]
Multiplying this by the Euler characteristics of \(U_{\mu_a}\) given 
in Lemma~\ref{lem:chi-U} proves the assertion.
\end{proof}

\begin{lemma}\label{lem:chi-U}
Let \(U_d\subset \mathbf P\bigl(\Sym^2(\C^{d})^*\bigr)\) be the space 
of nonsingular projective quadratic forms on \(\C^d\). We have
\[
  \chi(U_d)=
  \begin{cases}
    1,&d=1,2,\\
    0,&d\geq3.
  \end{cases}
\]
\end{lemma}

\begin{proof}
For $d=1$, the space $U_1$ is a point. For $d=2$, $U_2=\PP^2\setminus \Delta$ 
where \(\Delta\) is the discriminant conic $4bc-a^2=0$.  
Hence $\chi(U_2)=\chi(\PP^2)-\chi(\Delta)=1$.

Suppose $d\geq3$. Let \((e_1,\dots, e_d)\) be coordinates of \(\C^d\).
Consider a \(\mathbf G_m\)-action
\[
\mathbf G_m\times \C^d\to\C^d, \quad t\cdot e_i=t^{w_i}\cdot e_i
\]
with prescribed weights \((w_1,\dots, w_d)\). We may choose suitable 
\(w_i\) such that the weights \(-(w_i+w_j)\) on the coordinate monomials 
\(e^*_ie^*_j\) of \(\mathrm{Sym}^2 (\C^d)^*\) are pairwise distinct. 
Thus the fixed points in \(\PP(\mathrm{Sym}^2 (\C^d)^*)\) are exactly 
the coordinate monomials, which represent forms of rank at most two, 
so $U_d$ has no fixed point under the \(\mathbf G_m\)-action. 

Bia{\l}ynicki--Birula~\cite{BialynickiBirula} studied the fixed point 
formula of complete varieties with a \(\mathbf G_m\)-action and additive 
invariants. Hu extended their study and proved that if a reduced and 
irreducible algebraic variety \(X\) admits a \(\mathbf G_m\)-action, then
\(X\) is \(\mathbf G_m\)-equivalent to the fixed point locus 
\(X^{\mathbf{G}_m}\) in the sense of~\cite[Section~2]{Hu}.
As a consequence, 
\[
\chi_c(X)=\chi_c(X^{\mathbf{G}_m}).
\]
Therefore, in our cases,
\[
  \chi_c(U_d)=\chi_c(U_d^{\mathbf G_m})=0, \quad d\geq 3.
\]
\end{proof}

\subsection{Recurrence formula}

\begin{theorem}
\label{thm:local-global-zeta-function}
Let $\cC_{1,2}(n)$ denote the set of partitions of $n$ all of whose
parts are $1$ or $2$. For $n\geq 1$, the local and global topological 
zeta functions of the determinant function \(f\) in ~\eqref{eq:symmetric-determinant} are
\[
\begin{aligned}
  Z^{\mathrm{top}}_{f,0}(s) &=
  \sum_{\mu\in\cC_{1,2}(n)}
  \frac{n!}{2^{q(\mu)}}
  \prod_{a=1}^{\ell(\mu)}
  \frac{1}{s_a\left(s+\dfrac{s_a+1}{2}\right)}\\
  Z^{\mathrm{top}}_f(s) &= Z^{\mathrm{top}}_{f,0}(s)
\end{aligned}
\]
\end{theorem}

\begin{proof}
The local Denef--Loeser topological zeta function is
\[
Z^{\mathrm{top}}_{f,0}(s) \coloneqq \sum_{I\subset [n]} \chi(E^{\circ}_I\cap \pi^{-1}(0)) \prod_{i\in I}\frac{1}{\nu_i+sN_i}. 
\]
By Corollary~\ref{cor:affine-CQ}, the fiber $\pi^{-1}(0)=E_0$ is the 
zero section of the projection $q\colon X\to \CQ(V)$. Then the line 
bundle structure~\eqref{eqs:line-bdl-div-orb} implies that 
\[
  E_I^\circ\cap\pi^{-1}(0)=
  \begin{cases}
    D_{I\setminus\{0\}}^\circ,&0\in I\subset [n],\\
    \varnothing,&0\notin I\subset [n].
  \end{cases}
\]
Let \(\mu\) be the ordered partition of \(n\) associated to 
\(I_\mu=I\setminus\{0\}\). Then we have
\[
\chi(D^{\circ}_{I\setminus\{0\}})=
\begin{cases}
\dfrac{n!}{2^{q(\mu)}},
&\text{if all parts of } \mu \text{ are }1\text{ or }2,\\[2mm]
0,&\text{otherwise}.
\end{cases}
\]
Therefore the local zeta function is reduced to 
\[
Z^{\mathrm{top}}_{f,0}(s) \coloneqq \sum_{\mu\in\cC_{1,2}(n)} \dfrac{n!}{2^{q(\mu)}} \prod_{r\in \{0\}\cup I_{\mu}}\frac{1}{\nu_r+sN_r}.
\]
Recall the numerical data \((N_r, \nu_r) = (n-r, \frac{(n+1-r)(n-r)}{2})\).
Write $r=n-s_a\in I_{\mu}$ and \(0=n-s_k\) where \(k=\ell(\mu)\) is
the length of \(\mu\). Then the denominator
\[
  \nu_{n-s_a}+sN_{n-s_a}
  =
  s_a\left(s+\frac{s_a+1}{2}\right), \quad 1\le a\le \ell(\mu),
\]
yields the formula
\begin{equation}
\label{eqs:local-zeta-fun}
Z^{\mathrm{top}}_{f,0}(s)
=
\sum_{\mu\in\mathcal C_{1,2}(n)}
\frac{n!}{2^{q(\mu)}}
\prod_{a=1}^{\ell(\mu)}
\frac1{
s_a
\left(s+\dfrac{s_a+1}{2}\right)
}.
\end{equation}

The global topological zeta function is
\[
Z^{\mathrm{top}}_f(s) \coloneqq \sum_{I\subset [n]} \chi(E^{\circ}_I) \prod_{i\in I}\frac{1}{\nu_i+sN_i}.   
\]
It follows from Corollary~\ref{cor:affine-CQ} that \(E_I^\circ\) is a 
\(\mathbf C^*\)-bundle over \(D^{\circ}_I\) if \(0\notin I\), and 
\[
\chi(E_I^\circ)
=
\chi(\mathbf C^*)\chi(D_I^\circ)=0.
\]
The remaining sum over \(0\in I\) yields the same formula~\eqref{eqs:local-zeta-fun}.
\end{proof}

Write $Z_n(s)=Z^{\mathrm{top}}_{f,0}(s)$, with formal initial inputs
\(Z_0(s)=1, \,Z_1(s)=\frac{1}{s+1}\). By discussing the last term
of an ordered partition, we can obtain the recurrence
\begin{equation}
\label{eq:recurrence}
  Z_n(s)=
  \frac{Z_{n-1}(s)+\dfrac{n-1}{2}Z_{n-2}(s)}
  {s+\dfrac{n+1}{2}}
  \qquad(n\geq2).
\end{equation}
Indeed, if $\nu$ is a partition of $n-1$ and \(\mu\) is the partition 
of \(n\) obtained by adding $1$ to \(\nu\) as the last term, we have \(q(\mu)=q(\nu)\) and 
\(\ell(\mu)=\ell(\nu)+1\). It follows that
\[
\frac{n!}{2^{q(\mu)}}\prod_{a=1}^{\ell(\mu)}
\frac{1}{s_a\left(s+\dfrac{s_a+1}{2}\right)}=
\frac{n}{n(s+\frac{n+1}{2})}\cdot \frac{(n-1)!}{2^{q(\nu)}}
\prod_{a=1}^{\ell(\nu)}
\frac1{
s_a
\left(s+\dfrac{s_a+1}{2}\right)
}.
\]
If $\nu$ is a partition of $n-2$ and \(\mu\) is the partition 
of \(n\) obtained by adding $2$ to \(\nu\) as the last term, we have 
\(q(\mu)=q(\nu)+1\) and \(\ell(\mu)=\ell(\nu)+1\). We have
\[
 \frac{n!}{2^{q(\mu)}}\prod_{a=1}^{\ell(\mu)}
\frac{1}{s_a\left(s+\dfrac{s_a+1}{2}\right)}=
\frac{n(n-1)}{2n(s+\frac{n+1}{2})}\cdot \frac{(n-2)!}{2^{q(\nu)}}
\prod_{a=1}^{\ell(\nu)}
\frac1{
s_a
\left(s+\dfrac{s_a+1}{2}\right)
}.
\]
Hence, the sum over the partitions of \(n-1\) has the common factor 
\(\frac{1}{s+\frac{n+1}{2}}\), and the sum over the partitions of 
\(n-2\) has the common factor \(\frac{n-1}{2(s+\frac{n+1}{2})}\).
Summing these two classes yields~\eqref{eq:recurrence}.

\begin{theorem}
\label{thm:closed-form}
The expressions~\eqref{eq:Fodd} and~\eqref{eq:Feven}
\[
  F_{2m+1}(s)=\frac{(2m+1)!!}{(s+1)\prod_{j=1}^{m}(2s+2j+1)},\quad                
  F_{2m}(s)=\frac{(2m-1)!!(s+2m+1)}{(s+1)\prod_{j=1}^{m}(2s+2j+1)}
\]
are the solutions of the recurrence~\eqref{eq:recurrence} w.r.t the 
parity. Hence
\begin{equation}
\label{eq:closed-closed}
  Z_n(s)=F_n(s),
\end{equation}
and the pole set is exactly the one in \eqref{eq:actual-poles} with \(r=0\).
\end{theorem}

\begin{proof}
The assertion holds for $n=1$, and \eqref{eq:recurrence} gives
\[
  Z_2(s)=\frac{s+3}{(s+1)(2s+3)}=F_2(s).
\]
Then we can verify the expressions \(F_{2m+1}(s)\) and \(F_{2m}(s)\)
satisfy the following recurrences
\[
  F_{2m+1}=\frac{F_{2m}+mF_{2m-1}}{s+m+1}, \quad 
  F_{2m+2}=\frac{F_{2m+1}+\frac{2m+1}{2}F_{2m}}{s+\frac{2m+3}{2}}, \quad m\geq1
\]
\end{proof}

\begin{remark}
Let $x\in D$ have rank $r$. Lemma~\ref{lem:normal-slice} identifies 
the holomorphic germ of $f$ at $x$ with the determinant on
$\Sym^2(\C^{n-r})^*$ at the origin multiplying with a unit. 
Theorem \ref{thm:closed-form} therefore gives
\begin{equation}
\label{eq:local-zeta-closed}
Z^{\mathrm{top}}_{f,x}(s)=F_{n-r}(s).
\end{equation}
\end{remark}

\subsection{b-functions of prehomogeneous vector spaces}
\label{sec:b-function}
Let $(G,\rho,W)$ be an irreducible regular prehomogeneous vector space
over $\C$, with $G$ connected and reductive. $W$ has a dense $G$-orbit 
whose complement is the zeros of the relative invariant $F$ of $(G,\rho,W)$,
where $F(\rho(g)w)=\chi(g)F(w)$ for a character $\chi\colon G\to \C^*$. 
The dual representation $\rho^*$ on $W^*$ is again regular prehomogeneous 
and has a relative invariant $F^*$ with character $\chi^{-1}$. The 
\(b\)-function of $(G,\rho,W)$ is a monic polynomial \(b_F(s)\) with
$\deg b_F=\deg F$ that satisfies
\[
  F^*(\partial_1, \ldots, \partial_n)F^{s+1}=b_F(s)F^s;
\]
see~\cite[Section~1]{KimuraSatoZhu}. This coincides with the 
Bernstein-Sato polynomial of the relative invariant \(F\).

For $G=\GL(V)$ and $W=\Sym^2(V^*)$, take $\rho(g)A=gAg^{t}$. The dual 
action on $\Sym^2(V)$ is $\rho^*(g)B=g^{-t}Bg^{-1}$, and both relative 
invariants are determinants. Kimura computed the associated \(b\)-function
\cite[Section~2-2]{Kimura}.

\begin{theorem}
\label{thm:b-function}
For the determinant function \(f\) on $\Sym^2(\C^n)^*$, the 
Bernstein--Sato polynomial of \(f\) is
\begin{equation}
  \label{eq:global-b}
  b_f(s)=\prod_{j=1}^{n}\left(s+\frac{j+1}{2}\right).
\end{equation}
Suppose that \(x\in D\) is a point of rank \(r\). The local Bernstein-Sato
polynomial is
\begin{equation}
  \label{eqs:local-b-fun}
  b_{f,x}(s)=\prod_{j=1}^{n-r}(s+\frac{j+1}{2}).
\end{equation}
\end{theorem}
\begin{proof}
The global formula is cited above. By Lemma~\ref{lem:normal-slice}, 
locally $f=u\det(S)$, where $u$ is a holomorphic unit and $S$ is a
symmetric $(n-r)\times (n-r)$ matrix. Then the local \(b\)-function 
of \(f\) at \(x\) equals the local \(b\)-function of \(\det S\) at the 
origin, which leads to the local formula~\eqref{eqs:local-b-fun}.
\end{proof}

\begin{proof}[Proof of Theorem~\ref{thm:main}]
Compare the local topological zeta function formula 
\(Z^{\mathrm{top}}_{f,x}(s)\) in~\eqref{eq:local-zeta-closed}
with $b_{f,x}(s)$. We have
\[
  b_{f,x}(s)Z^{\mathrm{top}}_{f,x}(s)\in\C[s].
\]
The global assertion follows from $Z^{\mathrm{top}}_f=Z^{\mathrm{top}}_{f,0}$
and $b_f=b_{f,0}$.
\end{proof}

\subsection*{Data availability}
No data were generated or analysed in this manuscript.

\end{document}